\pdfoutput=1 
\documentclass{article}
\usepackage[latin1]{inputenc}
\usepackage{amsmath, amssymb, amsthm}
\usepackage{bbold}			
\usepackage{color}

\numberwithin{equation}{section}

\newtheorem{theo}{Theorem}[section]
\newtheorem{pr}[theo]{Proposition}

\newtheorem{lem}[theo]{Lemma}

\newtheorem{rmk}{Remark}[section]

\newcommand{\Z}{\mathbb{Z}}
\newcommand{\N}{\mathbb{N}}
\newcommand{\R}{\mathbb{R}}
\newcommand{\V}{\mathbb{V}}

\newcommand{\PP}{\mathbb{P}}
\newcommand{\EE}{\mathbb{E}}

\begin{document}

\title{A Quenched Functional Central Limit Theorem for Random Walks in Random Environments under $(T)_\gamma$}
\author{ \'Elodie Bouchet\footnotemark[1], Christophe Sabot \footnotemark[1]
 \\ and Renato Soares dos Santos\footnotemark[2] }

\footnotetext[1]{Institut Camille Jordan, CNRS UMR 5208, Universit\'e de Lyon, Universit\'e Lyon 1, 43, Boulevard du 11 novembre 1918, 69622 Villeurbanne, France. 
E-mail: bouchet@math.univ-lyon1.fr; sabot@math.univ-lyon1.fr. \vspace{-10pt} \\}
\footnotetext[2]{
 WIAS, Mohrenstr.\ 39, 10117 Berlin, Germany. E-mail: soares@wias-berlin.de.\\
This work was partially supported by the ANR project MEMEMO2.
}

\maketitle

\begin{abstract}
We prove a quenched central limit theorem for random walks in i.i.d.\ weakly elliptic random environments
in the ballistic regime.
Such theorems have been proved recently by Rassoul-Agha and Sepp\"al\"ainen in \cite{RaSe09}
and Berger and Zeitouni in \cite{BeZe08} under the assumption of large finite moments for the regeneration time.
In this paper, with the extra $(T)_{\gamma}$ condition of Sznitman we reduce the moment condition
to $\EE(\tau^2(\ln \tau)^{1+m})<+\infty$ for $m>1+1/\gamma$,  
which allows the inclusion of new non-uniformly elliptic examples such as Dirichlet random environments.

\end{abstract}

\section{Introduction}
\label{s:intro}

This paper is concerned with multidimensional random walks in random environments (RWRE) in the ballistic regime. 
We begin with a definition of the model, followed by a brief motivation and historical account.
We then state our results and give an outline of the remainder of the paper.

\subsection{Model}
\label{ss:model}

Let $ \mathcal{P} := \{ (p_z)_{z \in \Z ^d} : p_z \geq 0, \sum_{z} p_z = 1 \} $ be the simplex of all probability vectors on $\Z^d$. 
We call \emph{environment} any element $\omega := \{ \omega_x : x \in \Z^d \}$ of the environment space $ \Omega := \mathcal{P} ^{\Z^d} $. 
Fixed $\omega \in \Omega$, the random walk in environment $\omega$ starting from $x$ 
is defined as the Markov chain $ \{ X_n : n \geq 0 \} $ in $\Z^d$ with law $ P_{x,\omega} $ 
such that $ P_{x,\omega}( X_0 = x )=1 $ and
\[ P_{x,\omega}(X_{n+1} = y | X_n = x) = \omega_{x , y-x} \]
for each $x,y \in \Z^d$.

A random environment is specified by choosing a probability measure
$\PP$ on the environment space $\Omega$.
We will assume that $ \{ \omega_x : x \in \Z^d \} $ are i.i.d.\ under $\PP$. 
The distribution $P_{x,\omega}$ is called the \emph{quenched} law of the RWRE starting from $x$, 
and $ \PP_x: = \int P_{x,\omega} d\PP $ its \emph{averaged} or \emph{annealed} law. 
We denote by $ E _{x,\omega} $ and $ \EE _x $ the corresponding expectations.

Given a vector $ v_\star \in \R ^{d} \setminus \{0\}$, we say that the RWRE is \emph{transient in direction $v_\star$} if
\begin{equation}\label{e:transience}
\PP_0 \left( \lim_{n \to \infty} X_n \cdot v_\star = \infty\right) = 1.
\end{equation}
It is well known that, if \eqref{e:transience} is satisfied, it is possible to define \emph{regeneration times} $(\tau_k) _{k \in \N}$ 
that satisfy $ \sup _{n < \tau_i} X_n \cdot v_\star < X _{\tau_i} \cdot v_\star = \inf _{n \geq \tau_i} X_n \cdot v_\star $. 
These were first introduced by Sznitman and Zerner in \cite{SzZe99}, where their construction is detailed.

Another important type of hypotheses for the model are ellipticity assumptions. 
These are conditions on the positivity of $\omega_{0,z}$ for some collection of sites $z \in \Z^d$.
For example, the environment is called \emph{elliptic} if there exists
a basis $\{e_1, \ldots, e_d\}$ of $\Z^d$ such that $\omega_{0,\pm e_j} > 0$ a.s.\ for all $j = 1,\ldots,d$,
and \emph{uniformly elliptic} if there exists some $\epsilon > 0$ such that
$\omega_{0,\pm e_j} \ge \epsilon$ a.s.\ for all $j=1,\ldots, d$.
Milder conditions are also used in the literature;
in this case we call the environment ``weakly elliptic''.

\subsection{Motivation}
\label{ss:motivation}

The model described in Section~\ref{ss:model} has been the subject of intense study for several decades now.
In one dimension, it is currently very well understood but, in higher dimensions, 
important questions remain open despite many accomplishments.
A subclass of models for which several results are available is the so-called \emph{ballistic regime}. 
In this case, not only is the RWRE transient as in \eqref{e:transience}
but moves linearly with time in direction $v_\star$, often satisfying some prescribed deviation bounds.

One way to obtain such bounds is to require finite moments for the regeneration time.
Indeed, for $ d \geq 2 $, under \eqref{e:transience} and supposing $ \EE_0 \left[ \tau_1 \right] < \infty $, 
the RWRE satisfies a law of large numbers with a non-degenerate velocity 
(see Sznitman and Zerner \cite{SzZe99} and Zerner \cite{Ze02}) 
while, if $\EE_0 \left[ \tau_1^{2} \right] < \infty$, it also satisfies a 
functional central limit theorem (FCLT) under the annealed law (see Sznitman \cite{Sz00,Sz01}).

Under stricter conditions, one is able to prove
also a FCLT under the quenched law.
This has been proved for example by Rassoul-Agha and Sepp\"al\"ainen in \cite{RaSe09}
under weak ellipticity and $\EE_0 \left[ \tau_1^{176d+\varepsilon} \right] < \infty$,
and by Berger and Zeitouni in \cite{BeZe08} under uniform ellipticity and $\EE_0 \left[ \tau_1^{40} \right] < \infty$.

The aim of this article is to improve the moment assumptions on $\tau_1$ under an additional restriction: 
the ballisticity condition $(T)_\gamma$, first introduced by Sznitman in \cite{Sz01,Sz02}.
It is then enough to require $ \EE_0 \left[ \tau_1^{2} (\ln \tau_1 ) ^m \right] < + \infty $
with $m > 1 + 1/\gamma$.
This result offers no improvement in the uniformly elliptic case
since then condition $(T)_\gamma$ is known to imply finite moments of all orders for $\tau_1$.
However, without uniform ellipticity the latter implication might not true, 
and our extension can considerably improve the range of validity of the quenched FCLT. 
This is the case for example when the random environment has a product Dirichlet distribution, 
as discussed in Section~\ref{sec:Dirichlet} below.

\subsection{Main results}
\label{ss:results}

In the following, we assume the existence of $v_\star \in \R^d \setminus \{0\}$ such that the random walk in random environment satisfies (\ref{e:transience}).
We will also assume the following conditions:
\begin{enumerate}

\item[]
\textbf{Condition $(S)$.} The walk has bounded steps: there exists a finite, deterministic and positive constant $r_0$ such that $ \PP ( \omega _{0,z} = 0 ) = 1 $ for all $ |z| > r_0 $.

\item[]
\textbf{Condition $(R)$.} 
The set $ \mathcal{J} := \{ z : \EE ( \omega _{0,z} ) > 0 \}$ of admissible steps under $\PP$ satisfies $ \mathcal{J} \not\subset \R u $ for all $ u \in \R ^d $.
Furthermore, $ \PP ( \exists z : \omega_{0,0} + \omega_{0,z} = 1 ) < 1 $.

\end{enumerate}

Conditions $(S)$ and $(R)$ are exactly as stated in the article \cite{RaSe09} of Rassoul-Agha and Sepp\"al\"ainen. 
They hold for example in the case of nearest-neighbour random walks in elliptic random environments, such as Dirichlet random environments. 

Transience, $(S)$, $(R)$ and $ \EE_0 \left[ \tau_1 \right] < + \infty $ are enough to imply 
a strong law of large numbers with a velocity $ v \in \R^d \setminus \{ 0 \} $ 
(see Sznitman and Zerner \cite{SzZe99} and Zerner \cite{Ze02}), i.e.,
\[ \PP_0 \left( \lim _{n \to \infty} \frac{X_n}{n} = v \right) = 1 .\]
Define now a sequence of processes $ (B^{(n)}_t) _{t\geq0} $ by setting
\begin{equation}
B^{(n)}_t = \frac{X_{[nt]}-[nt]v}{\sqrt{n}}, \;\; t \ge 0,
\end{equation}
where $ [x] := \max \{ n \in \Z : n \leq x \}$ stands for the integer part of $x$. 
If in addition $ \EE_0 \left[ \tau_1^{2} \right] < + \infty $, 
then $B^{(n)}$ converges under the annealed law to a Brownian motion
with a non-degenerate covariance matrix (see Sznitman \cite{Sz00,Sz01}). 

Our key assumption is the ballisticity condition $(T)_\gamma$, introduced by Sznitman in \cite{Sz01,Sz02}:
\begin{enumerate}
\item[]
\textbf{Condition $(T)_\gamma$, with $ 0 < \gamma \leq 1 $.}
The walk is transient in direction $v_\star \neq 0$ and there exists $c > 0$ such that
\begin{equation}\label{condition-T}
\EE_0 \left[ \exp \left( c \sup_{1 \le n \le \tau_1} \|X_n\| ^\gamma \right) \right] < \infty .
\end{equation}
\end{enumerate}
All concrete examples where an annealed FCLT has been proved so far satisfy condition $(T)_\gamma$.

Our main result is the following:
\begin{theo}[Quenched Functional Central Limit Theorem]
\label{theo:quenchedFCLT}
Set $d \geq 2 $. We consider a random walk in an i.i.d.\ random environment. 
Assume that there is a direction $v_\star \in \R ^d \setminus \{ 0 \}$ where the transience condition (\ref{e:transience}) holds, 
and denote by $\tau_1$ the corresponding regeneration time.
Suppose that the walk satisfies conditions $(S)$, $(R)$ and $ (T)_\gamma $ for some $ 0 < \gamma \leq 1 $.
Also suppose that $ \EE_0 \left[ \tau_1^{2} (\ln \tau_1 ) ^m \right] < + \infty $ for some $ m > 1 + \frac{1}{\gamma}$. 
Then, for $\PP$-a.e.\ environment $\omega$, 
the process $B^{(n)}_t$ converges in law (as $ n \to \infty $) under $P_{0,\omega}$ to a Brownian motion with a deterministic, 
non-degenerate covariance matrix.
\end{theo}

Note that under a rather weak integrability condition on the environment at one site, denoted $(E')_0$, Campos and Ramirez 
proved in the
non-uniformly elliptic case that the $(T)_{\gamma}$ conditions are all equivalent to a weaker polynomial condition $(P)_M$, cf \cite{CaRa13}
for a precise result
(in the uniformly elliptic case, this is a result of Berger, Drewitz and Ram\'\i rez, \cite{BeDrRa12}). 
The polynomial condition
is not stated in the same terms as in formula (\ref{condition-T}), 
but in terms of exit probabilities of boxes.
Nevertheless, it can be shown to be
equivalent to an integrability condition on some polynomial moments of $ \sup_{1 \le n \le \tau_1} \|X_n\|$.
With respect to the quenched CLT, theorem 1.1 of \cite{CaRa13} tells us that if condition $(E')_0$ and condition $(P)_M$ are satisfied for $M>15d+5$ (cf \cite{CaRa13} for the definition), 
then condition $(T)'=\cap_{\gamma\in (0,1)} (T)_\gamma$ holds, reducing the condition of theorem  \ref{theo:quenchedFCLT}
to $ \EE_0 \left[ \tau_1^{2} (\ln \tau_1 ) ^m \right] < + \infty $ for some $m>2$.

Note also that sufficient conditions for the integrability of moments $\EE_0 \left[ \tau_1^{p} \right]$ 
have been given in two papers, by Bouchet, Ram\'\i rez, Sabot in \cite{BoRaSa14} and Fribergh, Kious in \cite{FrKi14}. 
These conditions involve the environment at one site or in a small box.

\begin{rmk}
We believe that condition $(T)_\gamma$ in theorem~\ref{theo:quenchedFCLT} above can be replaced by the condition
$$\EE_0 \left[ \left( \sup_{1 \le n \le \tau_1} \|X_n\| \right) ^p \right] < \infty \; \text{ for some } \, p \text{ large enough (depending on $d$)}$$
and a more restrictive moment condition on the renewal times. 
But in fact, as explained above, under a rather weak ellipticity condition, the condition above is
equivalent to the condition  $(T)_\gamma$, cf Theorem 1.1 in \cite{CaRa13}.
\end{rmk}

Our proof of theorem~\ref{theo:quenchedFCLT} is based on strategies and techniques used by Berger and Zeitouni in \cite{BeZe08}, and by Rassoul-Agha and Sepp\"al\"ainen in \cite{RaSe09}. The key differences compared to \cite{BeZe08} and \cite{RaSe09} are in the following two steps:
\begin{itemize}
\item We improve the key estimate of section 4 of \cite{BeZe08} by means of condition $(T)_\gamma$.
\item The construction of the joint regeneration times of section 7 of \cite{RaSe09} is modified so that condition $(T)_\gamma$ 
can be used to get better estimates on the number of intersections. 
\end{itemize}

\subsection{Outline}
\label{ss:outline}

The rest of the paper is organized as follows.
In Section~\ref{sec:Dirichlet}, 
we discuss Dirichlet random environments, which are examples 
where theorem~\ref{theo:quenchedFCLT} significantly improves previously known results.
The proof of theorem~\ref{theo:quenchedFCLT} 
is given in section~\ref{s:proof_theo_qFCLT} conditionally
on two auxiliary theorems.
The first auxiliary theorem reduces the problem to bounding
the expected number of intersections of two independent copies of the RWRE in the same random environment,
while the second provides the required bound.
Their proofs are given in sections~\ref{sec:boundbyintersec} and \ref{sec:boundonintersections}
and are based on section 4 of \cite{BeZe08} and section 7 of \cite{RaSe09}, respectively.

\section{An illustration: the case of Dirichlet environments}
\label{sec:Dirichlet}

In this section, we consider a particular case for which our theorem~\ref{theo:quenchedFCLT} significantly increases the understanding of the behavior: the case of random walks in Dirichlet random environments. 
This case is particularly interesting because it offers analytical simplifications, and because it is linked with reinforced random walks.
Indeed, the annealed law of a random walk in Dirichlet environment corresponds exactly to the law of a linearly directed-edge reinforced random walk (\cite{Pe88}, \cite{EnSa06,Sa11}).

In the following, we consider a nearest neighbor walk on $\Z^d$, $d\geq2$, and we note $ e_1, \dots, e_{2d} $ the canonical vectors with the convention $ e_{d+i} = - e_i $ for $ 1 \leq i \leq d $.

Given a set of positive real weights $ (\alpha _1, \dots , \alpha _{2d}) $, a random i.i.d. Dirichlet environment is a law on $\Omega$ constructed by choosing independently at each site $x \in \Z^d$ the values of $ \left( \omega _{ x, e_i } \right) _{i \in [\![1,2d]\!]} $ according to a Dirichlet law with parameters $ (\alpha _1, \dots , \alpha _{2d})$, i.e.  the law with density 
\[ \frac{ \Gamma \left( \sum _{i=1} ^{2d} \alpha _i \right) }{ \prod_{i=1} ^{2d} \Gamma \left( \alpha _i \right) } \left( \prod _{i=1} ^{2d} x_i ^{\alpha _i - 1} \right) dx_1 \dots dx_{2d-1} \]
on the simplex $\{(x_1,\dots, x_{2d}) \in ]0,1]^{2d}, \sum _{i=1} ^{2d} x_i = 1 \}$. Here $\Gamma$ stands for the usual Gamma function $ \Gamma (\beta) = \int _0 ^\infty t ^{\beta - 1} e ^{-t} dt$ , and $dx_1 \dots dx_{2d-1}$ represents the image of the Lebesgue measure on $ \R^ {2d-1} $ by the application $( x_1, \dots ,x_{2d-1} )  \to ( x_1, \dots , x_{2d-1}, 1- x_1 - \dots - x_{2d-1} ) $. It is straightforward that the law does not depend on the specific role of $ x_{2d}$. Note that Dirichlet environments are not uniformly elliptic: we can find no positive $c$ such that $ \PP ( \omega_{0,e_i} \geq c ) = 1 $. 

Set a Dirichlet law with fixed parameters $ ( \alpha_1, \dots, \alpha_{2d} ) $. 
Theorem 1 of \cite{BoRaSa14}  (see also Remark~3 and Section~1.3.2 therein) 
gives us that, if we assume $(T)_\gamma$ with $ 0 < \gamma \leq 1 $, then $ \EE_0 \left[ \tau_1^{p} \right] < + \infty $ is satisfied whenever
\[ \kappa :=  2 \left( \sum_{i=1} ^{2d} \alpha _i \right) - \max _{i=1,\dots, d} (\alpha _i + \alpha _{i+d}) > p .\]
On the other hand, condition $(T)_\gamma$ is satisfied for all $ 0 < \gamma \leq 1 $ whenever
\begin{equation}\label{e:condforT_Dirichlet}
\sum _{1 \leq i \leq d} | \alpha_i - \alpha_{i+d} | > 1.
\end{equation}
Indeed, it was shown by Tournier in \cite{To09} (see also Enriquez and Sabot \cite{EnSa06}) 
that \eqref{e:condforT_Dirichlet} implies Kalikow's condition,
which is in turn known to imply condition $(T) := (T)_1$; see Remark 2.5 (ii) in \cite{Sz01}.
The complete characterization of $(T)_\gamma$ in terms of the parameters of the Dirichlet law remains an open question, but we believe condition $(T)_\gamma$ should be satisfied if and only if $ \max _{1 \leq i \leq d} | \alpha_i - \alpha_{i+d} | > 0 $, 
i.e., we expect the hypotheses of Theorem~\ref{theo:quenchedFCLT} to hold as soon as $\kappa > 2$ and the walk is non-symmetric.

Theorem~\ref{theo:quenchedFCLT} thus gives us a quenched functional central limit theorem for Dirichlet environments as soon as the walk is transient, $\kappa > 2$
and \eqref{e:condforT_Dirichlet} holds.
This is a real improvement compared to the results of \cite{BeZe08} (that do not apply as the Dirichlet environment is not uniformly elliptic) or \cite{RaSe09} (that required $\kappa > 176 d$).

Also note that, in dimension $d\geq3$, $ \kappa > 1 $ implies the existence of an absolutely continuous invariant probability measure for the environment viewed from the particle (see Sabot \cite{Sa13}), which gives directly $ \EE_0 \left[ \tau_1 \right] < + \infty $ in the non-symmetric case. However, it gives no information on $ \EE_0 \left[ \tau_1^{p} \right] $ for other $ p < \kappa $.

\section{Proof of theorem~\ref{theo:quenchedFCLT}}
\label{s:proof_theo_qFCLT}

To prove theorem~\ref{theo:quenchedFCLT}, we will use a method introduced by Bolthausen and Sznitman in \cite{BoSz02}.
Fix $T \in \N$ and $ F : C([0,T], \R) \to \R $ a bounded function that is $1$-Lipschitz, i.e., 
such that for all $ f,g \in C([0,T], \R) $, $ | F(f) - F(g) | \leq \sup_{x \in [0,T]} | f(x) - g(x) | $. 
Let $ W^{(n)} _\cdot $ be the polygonal interpolation of $ \frac{k}{n} \to B^{(n)}_\frac{k}{n} , k \geq 0 $, 
and take $ b \in (1,2] $. 
Since the annealed FCLT holds by our assumption on $\tau_1$,
lemma~4.1 of \cite{BoSz02} reduces the problem to showing 
\[
\sum_{n=0} ^{+\infty} \text{Var} \left( E _{0, \omega} ( F( W^{([b^n])}_\cdot ) ) \right) < + \infty,
\]
where $ \text{Var} \left( E _{0, \omega} ( F( W^{(n)}_\cdot ) ) \right) = \left\| \EE _0 \left[ F(W^{(n)})\middle| \omega \right] - \EE _0 \left[ F(W^{(n)})\right] \right\|^2_2 $.
This will be accomplished via two theorems described next.

Let $Q_n$ be the number of intersections of two independent copies of the walk $X$ in the same random environment $\omega$ up to time $n-1$, 
i.e.,
\begin{equation}\label{defQn}
Q_n := | X_{[0,n)} \cap \tilde{X}_{[0,n)}|,
\end{equation}
where $X_{[0,n)} := \{X_0, \ldots, X_{n-1}\}$,
$\tilde{X}$ is an independent copy of $X$ defined in the same random environment
and $|\cdot|$ denotes the cardinality of a set.

In the following two theorems, 
we consider a random walk in an i.i.d.\ random environment in $d\ge 2$, and 
assume that there exist $v_\star \in \R^d\setminus \{0\}$ and $0 < \gamma \leq 1 $ such that the walk is transient in direction $v_\star$ and satisfies conditions $ (S), (R) $ and $(T)_\gamma $.

\begin{theo}
\label{theo:boundbyintersec}
Assume that $ \EE_0 \left[ \tau_1^{2} (\ln \tau_1 )^m \right] < + \infty $ for some $m >1 + \frac{1}{\gamma}$.
Then there exist an integer $K$ and,
for all $\delta \in (0,1)$, a positive constant $C$ such that
\begin{equation}\label{e:bbintersec}
\left\| \EE _0 \left[ F(W^{(n)})\middle| \omega \right] - \EE _0 \left[ F(W^{(n)})\right] \right\|^2_2
\leq C \left(  (\ln n) ^{-(m- \frac{1}{\gamma})} + n^{-(1-\delta)} 
\EE_0\left[Q_{Kn} \right] \right) \; \forall\; n \ge 2.
\end{equation}
\end{theo}

\begin{theo}
\label{theo:boundonintersections}
For all $ 0 < \eta < \frac{1}{2} $, we can find a constant $  0 < C_\eta < \infty $ depending only on $\eta$ such that, for all $ n \geq 1 $,
\[  \EE_0\left[Q_n \right] \leq C_\eta n ^{1 - \eta} .\]
\end{theo}

Taking for example $ \delta = \frac{1}{8} $ and $ \eta = \frac{1}{4} $, 
theorems~\ref{theo:boundbyintersec} and \ref{theo:boundonintersections} give a positive constant $C$ such that
\[
\sum_{n=0} ^{+\infty} \text{Var} \left( E _{0, \omega} ( F( W^{([b^n])}_\cdot ) ) \right) 
\leq \sum_{n=0} ^{+\infty} C \left( n^{- ( m- \frac{1}{\gamma})} + b^{-\frac{n}{8}} \right) < \infty 
\]
since $ m- \frac{1}{\gamma} > 1$. This proves theorem~\ref{theo:quenchedFCLT}.

\section{Proof of theorem~\ref{theo:boundbyintersec}}
\label{sec:boundbyintersec}
Before we proceed to the proof, we briefly recall the construction of the regeneration times in direction $v_\star$.
Let $\{ \theta_n : n \geq 1 \}$ be the canonical time shifts on $ ( {\Z^d} ) ^\N$ and, 
for $u \geq 0$, let $ T ^{v_\star} _u := \inf \{ n \geq 0: X_n \cdot v_\star \geq u \} $.
Set $ D ^{v_\star} := \min \{ n \geq 0 : X_n \cdot v_\star < X_0 \cdot v_\star \} $.
We define 
\[ S_0 := 0 , \quad  M_0 := X_0 \cdot v_\star ,\]
\[ S_1 := T ^{v_\star} _{M_0+1}, \quad R_1 := D ^{v_\star} \circ \theta_{S_1} + S_1, \quad M_1 := \sup \{ X_n \cdot v_\star : 0 \leq n \leq R_1 \} ,\]
and then we iterate for all $k \geq 1$,
\[ S_{k+1} := T ^{v_\star} _{M_k+1} , \quad R_{k+1} := D ^{v_\star} \circ \theta _{S_{k+1}} + S_{k+1}, \quad M_{k+1} := \sup \{ X_n \cdot v_\star : 0 \leq n \leq R_{k+1} \} .\]
We can now define the regeneration times as $ \tau_1 := \min \{ k \geq 1 : S_k <\infty , R_k = \infty \} $ and, for all $n \geq 1$, 
$ \tau _{n+1} := \tau_1 (X_\cdot) + \tau_n ( X_{\tau_1+\cdot} - X _{\tau_1} ) $.
Then $((X_{i}-X_{\tau_k})_{\tau_k \le i \le \tau_{k+1}}, \tau_{k+1}-\tau_k)$, $k \ge 1$ are i.i.d.\ under $\PP_0$ and distributed as 
$( (X_i)_{0 \le i \le \tau_1}, \tau_1)$ under $\PP_0(\cdot | D^{v_\star}=\infty)$.
In particular, for any $k \ge 0$,
\[
\EE_0 \left[ f\left( (X_i - X_{\tau_k})_{i \ge \tau_k}, (\tau_{k+i}-\tau_k)_{i \in \N} \right)\right]
\le C \EE_0 \left[ f\left( (X_i)_{i \ge 0}, (\tau_i)_{i \in \N } \right) \right]
\]
for some constant $C > 0$ and any measurable non-negative function $f$.
Another observation is that, with this construction,
\[
 \inf_{n \ge \tau_k} X_n \cdot v_\star \ge \sup_{n \le \tau_{k-1}} X_n \cdot v_\star +1, \;\;\; k \ge 1.
\]

We now come to the proof of theorem~\ref{theo:boundbyintersec}. 
Recall that $F$ is a $1$-Lipschitz function, and that by assumption $(S)$ the walk $X$ has steps bounded by $r_0 \in \N$. We define, for $k \in \N$,
\begin{equation}
\mathcal{G}^{(n)}_k := \sigma\{ \omega_x \colon\, \|x\|_{\infty} \le r_0 T n \text{ and } x \cdot v_\star < k\},
\end{equation}
and
\begin{equation}
\begin{aligned}
\Delta^{(n)}_1 & := \EE _0 \left[ F(W^{(n)}) \middle| \mathcal{G}^{(n)}_1\right] - \EE _0 \left[ F(W^{(n)}) \right], \\
\Delta^{(n)}_k & := \EE _0 \left[ F(W^{(n)}) \middle| \mathcal{G}^{(n)}_k\right] - \EE _0 \left[ F(W^{(n)}) \middle| \mathcal{G}^{(n)}_{k-1}\right], \;\; k \ge 2.
\end{aligned}
\end{equation}
The $\Delta^{(n)}_k$ are martingale increments and, since $X$ has bounded steps, we can find 
an integer $c_0 \ge r_0 T$
such that:
\begin{equation}
\EE _0 \left[ F(W^{(n)})\middle| \omega \right] - \EE _0 \left[ F(W^{(n)})\right] = \sum_{k=1}^{c_0 n} \Delta^{(n)}_k .
\end{equation}
By the martingale property,
\begin{equation}\label{e:estim0}
\left\| \EE _0 \left[ F(W^{(n)})\middle| \omega \right] - \EE _0 \left[ F(W^{(n)})\right] \right\|_2^2
= \sum_{k=1}^{ c_0 n } \left\|\Delta^{(n)}_k \right\|_2^2,
\end{equation}
therefore we only need to study the $L^2$ norm of $\Delta^{(n)}_k$.

To that end, let $ h_k = T ^{v_\star} _{ k-1 } $ be the hitting time of the level $ k-1 $ in direction $v_\star$ and let $\hat{\tau}_k$ be the second regeneration time strictly larger than $h_k$. We do not take the first regeneration time because we need to make sure that $ ( X _{\hat{\tau}_k} - X_{h_k} ) \cdot v_\star \geq 1 $. Since, for all $i$, $ ( X _{\tau_{i+1}} - X _{\tau_{i}} ) \cdot v_\star \geq 1 $, taking the second regeneration times ensures this.

Define $W^{(n,k)}$ as the analogous of $W^{(n)}$ for the path obtained by concatenation of $(X_i)_{0 \le i \le h_{k}}$ with $(X_{\hat{\tau}_k + i}-X_{\hat{\tau}_k})_{i \ge 1}$. Note that, since $X$ has bounded steps,
\begin{equation}\label{e:bounddifpaths}
\sup_{t \ge 0} \left|W^{(n)}_t - W^{(n,k)}_t \right| \le \frac{ r_0 |\hat{\tau}_k-h_k| }{\sqrt{n}}.
\end{equation}
Moreover, $ ( X _{\hat{\tau}_k} - X_{h_k} ) \cdot v_\star \geq 1 $, $W^{(n,k)}$ is independent of $\sigma(\omega_x \colon\, k-1 \le x \cdot v_\star < k )$, and hence
\begin{equation}\label{apres_concat}
\Delta^{(n)}_{k} = \EE _0 \left[ F(W^{(n)}) - F(W^{(n,k)}) \middle| \mathcal{G}^{(n)}_{k}\right] - \EE _0 \left[ F(W^{(n)}) - F(W^{(n,k)}) \middle| \mathcal{G}^{(n)}_{k-1}\right] \; \text{ a.s.}
\end{equation}

Fix now an integer $K \ge 4 c_0 (\EE_0[\tau_1] \vee \EE_0[\tau_2-\tau_1])$ and a number $M_n > 0$. 
We will partition on the event $A^{(n)}_k:=\{\hat{\tau}_k-h_k \le M_n\} \cap \{\hat{\tau}_k \le Kn\}$ and its complement.
Using the Lipschitz property of $F$ and \eqref{e:bounddifpaths}--\eqref{apres_concat}, we obtain
\begin{align}\label{e:estim1}
\EE _0 \left[|\Delta^{(n)}_{k}|^2\right]
& \le \left\| \EE _0 \left[ F(W^{(n)}) - F(W^{(n,k)}), A^{(n)}_k \middle| \mathcal{G}^{(n)}_{k} \right] - \EE _0 \left[ F(W^{(n)}) - F(W^{(n,k)}), A^{(n)}_k \middle| \mathcal{G}^{(n)}_{k-1}\right] \right\|_2^2 \nonumber\\
& + 2 r_0 
\left\{ \EE _0 \left[\frac{\left|\hat{\tau}_k -h_k \right|^2}{n}, \hat{\tau}_k-h_k > M_n \right] 
+ \EE _0 \left[\frac{\left|\hat{\tau}_k -h_k \right|^2}{n}, \hat{\tau}_k > K n \right] \right\}.
\end{align}

Now we sum \eqref{e:estim1} on $k$, starting with the second terms. 
Let $\check{\tau}_{k}$ be the largest regeneration time smaller than $h_k$.
Using $ \hat{\tau} _{c_0 n} \leq \tau _{c_0 n + 1} $ we write
\begin{align}\label{e:estim4}
& \sum_{k=1}^{c_0 n} \EE _0 \left[\frac{\left|\hat{\tau}_k -h_k \right|^2}{n}, \hat{\tau}_k -h_k> M_n \right] 
 \le \frac{1}{ n (\ln M_n) ^{m - \frac{1}{\gamma}} } \sum_{k=1}^{c_0 n} \EE _0 \left[  \left| \hat{\tau}_k - \check{\tau}_k \right|^2 (\ln  \left| \hat{\tau}_k - \check{\tau}_k \right|)^{m - \frac{1}{\gamma}} \right] \nonumber \\
& \le \frac{1}{ n (\ln M_n)^{m - \frac{1}{\gamma}} } \sum_{k=1}^{c_0 n} \EE _0 \left[ |\tau_{k+1} - \tau_{k-1}|^2  (\ln |\tau_{k+1} - \tau_{k-1}| ) ^{m - \frac{1}{\gamma}} |(X_{\tau_k}-X_{\tau_{k-1}}) \cdot v_\star|  \right] \nonumber \\
& \le \frac{C}{ (\ln M_n)^{m - \frac{1}{\gamma}} } \EE _0 \left[ 
\tau_{2}^2  (\ln \tau_{2} ) ^{m - \frac{1}{\gamma}}
\|X_{\tau_{1}}\|  \right].
\end{align}
We claim that
\begin{equation}\label{e:holder}
\EE_0 \left[ \tau_2^2  (\ln \tau_2 ) ^{m - \frac{1}{\gamma}} \|X_{\tau_{1}}\|  \right] < \infty.
\end{equation}
Indeed, for integers $k \ge 2$ write
\begin{align}\label{estim_ineq1}
& \PP_0 \left( \tau_2^2 (\ln \tau_2)^{m-1/\gamma} \|X_{\tau_1}\| \ge k \right) \nonumber\\
& \quad \qquad \le \PP_0 \left(\|X_{\tau_1}\| \ge (\ln k /c)^{1/\gamma} \right) + \PP \left( \tau_2^2 (\ln \tau_2 )^{m-1/\gamma} \ge c^{1/\gamma} k (\ln k)^{-1/\gamma} \right).
\end{align}
The first term in \eqref{estim_ineq1} is summable by condition $(T)_\gamma$, so we only need to control the second. Let
\begin{equation}\label{deff}
f(x) := x (\ln x)^{1/\gamma}, \;\; x \in [e,\infty).
\end{equation}
Then $f$ is non-decreasing and
\[ f \left(c^{1/\gamma} k (\ln k)^{-1/\gamma} \right) \ge c_1 k\]
for some constant $c_1 > 0$ and all large enough $k \in \N$.
Furthermore, for any $x \ge e$,
\[ f \left(x^2 (\ln x)^{m-1/\gamma} \right) \le c_2 x^2 (\ln x)^m \]
for some other constant $c_2 > 0$.
Therefore, for large enough $k$ the second term in \eqref{estim_ineq1} is at most
\begin{equation}\label{estim_ineq2}
\PP_0 \left(f(\tau_2^2 (\ln \tau_2)^{m-1/\gamma}) \ge f(c^{1/\gamma} k (\ln k)^{1/\gamma}) \right)
\le \PP_0 \left( \tau_2^2 (\ln \tau_2 )^{m} \ge \frac{c_1}{c_2} k\right),
\end{equation}
which is summable since $\EE_0 \left[ \tau_2^2 (\ln \tau_2)^m \right] < \infty$ 
by the regeneration structure and our assumption on $\tau_1$.
To see this, note that, since the function $f(x) = x^2 (\ln x)^m$ is increasing on $[1,\infty)$,
$f(a+b) \le f(2 a\vee b) \le f(2a) + f(2b)$, and $f(2x) = 4 x^2 (\ln 2 + \ln x)^m$. This finishes the proof of \eqref{e:holder}.

To control the sum of the third terms in \eqref{e:estim1},
write, analogously to \eqref{e:estim4},
\begin{align}\label{e:estim6}
\sum_{k=1}^{c_0 n} \EE _0 \left[\frac{\left|\hat{\tau}_k -h_k \right|^2}{n}, \hat{\tau}_k > K n \right] 
& \le \frac{C}{n} \sum_{k=1}^{c_0 n} \EE _0 \left[ |\tau_{k+1} - \tau_{k-1}|^2\|X_{\tau_{k}}-X_{\tau_{k-1}}\|, \tau_{k+1} > K n  \right] \nonumber \\
& \le \frac{C}{n} \sum_{k=1}^{c_0 n} \Big\{ \EE _0 \left[ |\tau_{k+1} - \tau_{k-1}|^2 \|X_{\tau_{k}}-X_{\tau_{k-1}}\|, \tau_{k+1}-\tau_{k-1} > \frac{K n}{2}  \right]  \nonumber\\
& \qquad \quad \; \; + \EE _0 \left[ |\tau_{k+1} - \tau_{k-1}|^2 \|X_{\tau_{k}}-X_{\tau_{k-1}}\|, \tau_{k-1} > \frac{K n}{2} \right] \Big\} \nonumber\\
& \le \frac{C}{(\ln n)^{m-1/\gamma}} + C \EE_0 \left[|\tau_2|^2  \|X_{\tau_1}\|\right] \PP_0 \left( \tau_{c_0 n} > \frac{Kn}{2} \right),
\end{align}
where the last inequality is justified as follows: for the first term, perform a calculation similar to \eqref{e:estim1} and, for the second, use the regeneration property and the fact that $k-1 \le c_0 n$.
From \eqref{e:holder}, we obtain $\EE_0 \left[|\tau_2|^2 \| X_{\tau_1} \|  \right] < \infty$.
Moreover, since $\tau_{c_0 n}$ is a sum of $c_0n$ independent random variables with bounded second moment 
and first moment bounded by $K/(4c_0)$,
we have
\begin{equation}\label{e:estim7}
\PP_0 \left( \tau_{c_0 n} > \frac{Kn}{2} \right) 
\le \PP_0 \left( \tau_{c_0 n} - \EE_0[\tau_{c_0 n }] > \frac{K n }{4}\right) \le \frac{C}{n}.
\end{equation}

Now set
\begin{equation}
\tilde{\Delta}^{(n)}_k := \EE _0 \left[ F(W^{(n)}) - F(W^{(n,k)}), A^{(n)}_k \middle| \mathcal{G}^{(n)}_{k} \right] - \EE _0 \left[ F(W^{(n)}) - F(W^{(n,k)}), A^{(n)}_k \middle| \mathcal{G}^{(n)}_{k-1}\right].
\end{equation}
To control the sum on $k$ of the first terms in \eqref{e:estim1} means to control the sum of $\|\tilde{\Delta}^{(n)}_k\|_2^2$. 
We decompose these terms as follows.
Let
\begin{equation}
\begin{aligned}
\mathcal{H}_1^{(n)} & := \{x \colon \|x\|_{\infty} \le r_0 T n \text{ and } x \cdot v_\star < 1\}, \\
\mathcal{H}_k^{(n)}& := \{x \colon \|x\|_{\infty} \le r_0 T n \text{ and } k-1 \le x \cdot v_\star < k\}, \;\; k \ge 2.
\end{aligned}
\end{equation} 
Write $\mathcal{H}^{(n)}_k = \{z_1, \ldots, z_{N}\}$ where $N:=|\mathcal{H}^{(n)}_k|$,
and let
\begin{equation}
\begin{aligned}
\mathcal{G}^{(n)}_{k,0} & := \mathcal{G}^{(n)}_{k-1}, \\
\mathcal{G}^{(n)}_{k,j} & := \mathcal{G}^{(n)}_{k-1} \vee \sigma\{ \omega_{z_i} \colon\, i \le j\}, \;\; j = 1, \ldots, N.
\end{aligned}
\end{equation}
We have
\begin{equation}\label{e:2nddecomposition}
\tilde{\Delta}_k^{(n)} = \sum_{j = 1}^{N} \tilde{\Delta}^{(n)}_{k,j}
\end{equation}
where
\begin{equation}\label{def_tildedeltankj}
\tilde{\Delta}^{(n)}_{k,j} := \EE _0 \left[ F(W^{(n)}) - F(W^{(n,k)}), A^{(n)}_k \middle| \mathcal{G}^{(n)}_{k,j} \right] - \EE _0 \left[ F(W^{(n)}) - F(W^{(n,k)}), A^{(n)}_k \middle| \mathcal{G}^{(n)}_{k,j-1}\right]
\end{equation}
are still martingale increments. Thus
\begin{equation}\label{e:estim2}
\|\tilde{\Delta}^{(n)}_k \|_2^2 := \sum_{j=1}^N \|\tilde{\Delta}^{(n)}_{k,j} \|_2^2.
\end{equation}

Let $h_{z_j}$ be the hitting time of a point $z_j \in \mathcal{H}^{(n)}_k$. 
Note that, on $A^{(n)}_k$, if
$h_{z_j} \ge Kn$ then $h_{z_j} = \infty$,
and on the latter event
the integrands in \eqref{def_tildedeltankj} do not depend on $\omega_{z_j}$. 
Hence

\begin{equation}
\begin{aligned}
\tilde{\Delta}^{(n)}_{k,j} 
& = \EE _0 \left[ F(W^{(n)}) - F(W^{(n,k)}), A^{(n)}_k, h_{z_j}  < Kn \,\middle|\, \mathcal{G}^{(n)}_{k,j} \right] \\
& - \EE _0 \left[ F(W^{(n)}) - F(W^{(n,k)}), A^{(n)}_k, h_{z_j} < Kn \,\middle|\, \mathcal{G}^{(n)}_{k,j-1}\right].
\end{aligned}
\end{equation}
By the Lipschitz property of $F$, \eqref{e:bounddifpaths} and the definition of $A^{(n)}_k$, we have
\begin{equation}
|\tilde{\Delta}^{(n)}_{k,j}| \le r_0 \frac{M_n}{\sqrt{n}} \left( \mathbb{P} _0 \left( h_{z_j} < Kn \,\middle|\, \mathcal{G}^{(n)}_{k,j}\right) + \mathbb{P} _0 \left( h_{z_j} < Kn \,\middle|\, \mathcal{G}^{(n)}_{k,j-1}\right) \right),
\end{equation}
and thus, by the Cauchy-Schwartz inequality,
\begin{equation}\label{e:estim3}
\|\tilde{\Delta}^{(n)}_{k,j}\|_2^2 
\le 2 r_0 ^2 \frac{M^2_n}{n} \EE _0 \left[ \mathbb{P} _0 \left( h_{z_j} < Kn \,\middle|\, \omega \right)^2 \right].
\end{equation}
Summing \eqref{e:estim3} on $j$ and $k$ and using \eqref{e:estim2}, we get
\begin{align}\label{e:estim8}
\sum_{k=1}^{c_0 n} \|\tilde{\Delta}^{(n)}_k\|_2^2
& \le C \frac{M_n^2}{n} \sum_{k=1}^{c_0 n} \sum_{z \in \mathcal{H}^{(n)}_k} \EE _0 \left[ \mathbb{P} _0 \left( h_{z} < Kn \,\middle|\, \omega \right)^2 \right] \nonumber\\
& = C \frac{M_n^2}{n}\EE_0 \left[ Q_{Kn} \right],
\end{align}
where $Q_n$ is the number of intersections of two independent copies of the walk $X$ in the same random environment as defined in \eqref{defQn}.

Finally, gathering the results in \eqref{e:estim0}, \eqref{e:estim1}--\eqref{e:holder}, \eqref{e:estim6}--\eqref{e:estim7} and \eqref{e:estim8}, we conclude
\begin{align}\label{e:estim5}
\left\| \EE _0 \left[ F(W^{(n)})\middle| \omega \right] - \EE _0 \left[ F(W^{(n)})\right] \right\|_2^2
\le C \left\{ \frac{1}{(\ln M_n)^{m - \frac{1}{\gamma}}} + \frac{1}{(\ln n)^{m - \frac{1}{\gamma}}} + \frac{1}{n} 
+ \frac{M^2_n}{n} \EE _0 \left[ Q_{Kn}\right] \right\}.
\end{align}
Taking $M_n = n^{\delta/2}$, we obtain theorem~\ref{theo:boundbyintersec}.

\section{Proof of theorem~\ref{theo:boundonintersections}}
\label{sec:boundonintersections}

We want to bound $\EE_0\left[Q_n \right]$, where $Q_n$ represents the number of intersections of two independent copies of $X$ in the same random environment $\omega$ up to time $n$. We note $X$ and $\tilde{X}$ the two independent walks driven by a common environment, then (recall \eqref{defQn})
\[ \EE_0 \left[Q_n \right] = \EE _{0,0} \left[ | X_{[0,n)} \cap \tilde{X}_{[0,n)} | \right] .\]

In this section, we will reduce to $ v_\star = e_i $ for convenience. This is possible because we will not make use of the integrability condition on $\tau_1$, and because the $(T)_\gamma$ hypothesis still holds for all all $e_i$ such that  $ v_\star \cdot e_i > 0 $ thanks to the following result:

\begin{pr}[Theorem 2.4 of \cite{CaRa13}]
Consider a RWRE in an elliptic i.i.d.\ environment. Let $ l \in \R ^d \setminus \{ 0 \} $. Then for all $ 0 < \gamma \leq 1 $ the following are equivalent.
\begin{itemize}
\item[(i)] Condition $ (T)_\gamma $ is satisfied in direction $l$.
\item[(ii)] There is an asymptotic direction $v$ such that $ l \cdot v > 0 $ and for every $l'$ such that $ l' \cdot v > 0 $ one has that $ (T)_\gamma $ is satisfied in direction $l'$.
\end{itemize}
\end{pr}

\begin{proof}
This result appears as theorem 2.4 of \cite{CaRa13}. It is primarily a consequence of theorem 1 of \cite{Si07}, that gives the existence of an asymptotic direction under $ (T)_\gamma $, and of theorem 1.1 of \cite{Sz02}. Theorem 1.1 of \cite{Sz02} is stated in the uniformly elliptic case, but the proof does not depend on it.
\end{proof}

We define the backtracking times $\beta$ and $ \tilde{\beta} $ for the walks $X$ and $ \tilde{X} $ as $\beta = \inf \{ n \geq 1 : X_n \cdot v_\star < X_0 \cdot v_\star \} $ and $ \tilde{\beta} = \inf \{ n \geq 1 : \tilde{X}_n \cdot v_\star < \tilde{X}_0 \cdot v_\star \} $. When the walks are on a common level, their difference lies in the hyperplane $ \V_d = \{ z \in \Z ^d : z \cdot v_\star = 0 \}$.

\subsection{Construction of joint regeneration times}

Lemma~7.1 of \cite{RaSe09} gives us that from a common level, there is a uniform positive probability $\eta$ for simultaneously never backtracking:

\begin{lem}[Lemma~7.1 of \cite{RaSe09}] \label{lem:backtrack}
Assume $ v_\star $ transience and the bounded step hypothesis (S). Then
\[ \eta = \inf _{x-y \in \V_d} \PP_{x,y} ( \beta \wedge \tilde{\beta} = \infty ) > 0 .\]
\end{lem}

We now introduce some additional notations. Set $ \gamma _l = \inf \{ n \geq 0 : X_n \cdot v_\star \geq l \} $ and $ \tilde{\gamma} _l = \inf \{ n \geq 0 : \tilde{X}_n \cdot v_\star \geq l \} $ the reaching times of level $l$ in direction $ v_\star $.

Let $h$ be the following greatest common divisor:
\begin{equation}
\label{e:defh}
h := \text{gcd} \{ l \geq 0 : \PP ( \exists n : X_n \cdot v_\star = l ) > 0 \} .
\end{equation}
Following the steps of \cite{RaSe09}, we get the following bound on joint fresh levels of two walks reached without backtracking:

\begin{lem}[Lemma~7.4 of \cite{RaSe09}] \label{lem:defl2}
There exists a finite $l_2$ such that we can find a constant $ c > 0 $ satisfying: uniformly over all $x$ and $y$ such that $ x \cdot v_\star, y \cdot v_\star \in [0, r_0 |v_\star|] \cap h \Z $,
\[ \PP _{x,y} \left( \exists i: ih \in [ 0 , l_2 h ] , X _{\gamma _{ih}} \cdot v_\star = \tilde{X} _{\tilde{\gamma} _{ih}} \cdot v_\star = ih, \beta > \gamma _{ih} , \tilde{\beta} > \tilde{\gamma} _{ih} \right) \geq c .\]
\end{lem}

We denote by $ \tau_k $ and $ \tilde{\tau}_k $ the regeneration times in direction $v_\star$ for the walks $X$ and $\tilde{X}$ (see \cite{SzZe99} for their explicit construction). We will now define the joint regeneration level of the two walks. For this, we cannot use the construction of \cite{RaSe09}, because it would not give a condition similar to condition $(T)_\gamma$ for the joint regeneration times, and we need such a condition later in the proof.

We then adapt their method, and begin by defining some new walks. They consist mainly in a time-change of the walks $X$ and $\tilde{X}$, that "stops" the walk that is the most advanced in direction $v_\star$, until the other walk outdistances it. Concretely, we construct the walks $ \underline{X} $ and $ \underline{\tilde{X}} $ and the times $ x_k $ and $\tilde{x}_k$ as follows: $ \underline{X}_0 = X_0 $, $ \underline{\tilde{X}}_0 = \tilde{X}_0 $, $ x_0 = \tilde{x}_0 = 0 $ and
\[ \left\{
\begin{array}{rcl}
x_{k+1} & = & x_k + \mathbb{1} _{ \{ \max _{i \leq k} ( \underline{X}_i \cdot v_\star ) \leq \max _{i \leq k} ( \underline{\tilde{X}}_i \cdot v_\star ) \} }  \\
\tilde{x}_{k+1} & = & \tilde{x}_k + \mathbb{1} _{ \{ \max _{i \leq k} ( \underline{X}_i \cdot v_\star ) > \max _{i \leq k} ( \underline{\tilde{X}}_i \cdot v_\star ) \} } \\
\underline{X}_{k+1} & = & X_{x_{k+1}}  \\
\underline{\tilde{X}}_{k+1} & = & \tilde{X} _{\tilde{x}_{k+1}}
\end{array}
\right.
.\]
Note that by construction, only one of the walks $ \underline{X} $ and $ \underline{\tilde{X}} $ moves at each step: at the $n$-th step, only the walk which corresponds to the smaller value between $ \max _{i \leq n} \{ \underline{X}_i \cdot v_\star \} $ and $ \max _{i \leq n} \{ \underline{\tilde{X}}_i \cdot v_\star \} $ moves.

We can now adapt the construction of \cite{RaSe09} to those new walks. Set $ \underline{\gamma} _l = \inf \{ n \geq 0 : \underline{X}_n \cdot v_\star \geq l \} $, $ \underline{\tilde{\gamma}} _l = \inf \{ n \geq 0 : \underline{\tilde{X}}_n \cdot v_\star \geq l \} $, $\underline{\beta} = \inf \{ n \geq 1 : \underline{X}_n \cdot v_\star < \underline{X}_0 \cdot v_\star \} $ and $ \underline{\tilde{\beta}} = \inf \{ n \geq 1 : \underline{\tilde{X}}_n \cdot v_\star < \underline{\tilde{X}}_0 \cdot v_\star \} $. We suppose that $ \underline{X} $ and $ \underline{\tilde{X}} $ start on a common level $ \lambda_0 \in h \Z $ (where $h$ is as in (\ref{e:defh})). We then define
\[ J = 
\left\{
\begin{array}{ll}
\sup \{ \underline{X}_i \cdot v_\star , \underline{\tilde{X}}_i \cdot v_\star : i \leq \underline{\beta} \wedge \underline{\tilde{\beta}} \} + h  & \text{ if } \underline{\beta} \wedge \underline{\tilde{\beta}} < \infty \\
\infty & \text{ if } \underline{\beta} \wedge \underline{\tilde{\beta}} = \infty
\end{array}
\right.
 \]
and 
\[ \lambda = 
\left\{
\begin{array}{ll}
\inf \{ l \geq J : \underline{X} _{\underline{\gamma}_l} \cdot v_\star = \underline{\tilde{X}} _{\underline{\tilde{\gamma}}_l} \cdot v_\star = l  \} & \text{ if } J < \infty \\
\infty & \text{ if } J = \infty
\end{array}
\right.
 .\]
In the case $ \lambda < \infty $, $ \lambda $ represents the first common fresh level after one backtrack. The case $ \lambda = \infty $ means that neither walk backtracked. We then define $ \lambda_1 $ the first common fresh level strictly after $\lambda_0$:
\[ \lambda_1 := \inf \{ l \geq \lambda_0 + h : \underline{X} _{\underline{\gamma}_l} \cdot v_\star = \underline{\tilde{X}} _{\underline{\tilde{\gamma}}_l} \cdot v_\star = l \} .\]
Lemma~\ref{lem:defl2} (which adapts immediately to the walks $\underline{X}$ and $\underline{\tilde{X}}$) ensures that $ \lambda_1 < \infty $.

We then construct $\lambda_n $ recursively for $n \geq 2$ as follows: if $ \lambda_{n-1} < +\infty $, we set
\[ \lambda_n := \lambda \circ \theta ^{ \underline{\gamma} _{\lambda_{n-1}} , \underline{\tilde{\gamma}} _{\lambda_{n-1}} }  \]
where $ \theta ^{m,n} $ represents the shift of sizes $m$ and $n$ on the pairs of paths.

For all $ n \geq 1 $, we say there is a joint regeneration at level $\lambda_n$ if $ \lambda_{n+1} = \infty $. It gives the following definition of the first joint regeneration level:
\[ \Lambda := \sup \{ \lambda_n : \lambda_n < \infty \} .\]

By construction, lemmas~\ref{lem:backtrack} and \ref{lem:defl2} also give that $ \Lambda < \infty $ a.s.. It allows to define the joint regeneration times for the walks $X$ and $\tilde{X}$:
\[ ( \mu_1 , \tilde{\mu}_1 ) = ( \gamma _\Lambda , \tilde{\gamma} _\Lambda ) .\]
We then get recursively the sequence:
\[ ( \mu_{i+1} , \tilde{\mu}_{i+1} ) = ( \mu_i , \tilde{\mu}_i ) + ( \mu_1 , \tilde{\mu}_1 ) \circ \theta ^{ \mu_i , \tilde{\mu}_i } .\]

\begin{rmk}
We could also define the joint regeneration times $ ( \mu_k, \tilde{\mu}_k ) $ as $ \mu_0 = \tilde{\mu}_0 = 0 $ and:
\begin{equation}\label{e:defmu}
( \mu_{k+1}, \tilde{\mu}_{k+1} )  := \inf_{n,m} \{ (\tau_n, \tilde{\tau}_m ): \tau_n > \mu_k ,   \tilde{\tau}_m > \tilde{\mu}_k , X _{\tau_n} \cdot v_\star = \tilde{X} _{\tilde{\tau}_m} \cdot v_\star \} ,
\end{equation}
and the first joint regeneration level as:
\[ \Lambda = X_{\mu_1} \cdot v_\star = \tilde{X} _{\tilde{\mu}_1} \cdot v_\star .\]
Those two definitions are equivalent, the interest of our previous construction by stages being that it will be easier to handle in the proofs. It also appears that the first joint regeneration level corresponds to the first level for which both walks $X$ and $\tilde{X}$ hit this level and regenerate.
\end{rmk}

We will now show that those joint regeneration times give us a bound similar to condition~$(T)_\gamma$. This will be the aim of the two following lemmas.

\begin{lem}
\label{lem:boundLambda}
For all $ m , p \geq 1 $, there exists a constant $ 0 < C_p < \infty $ depending only on $p$ such that:
\[ \sup _{ x,y \in \V_d } \PP_{x,y} ( \Lambda > m ) \leq C_p m ^{-p} .\]
\end{lem}

\begin{proof}

The proof is very similar to the proof of lemma~7.5 in \cite{RaSe09}. We first remark that the walks $ \underline{X} $ and $ \underline{\tilde{X}} $ follow exactly the same paths as $X$ and $\tilde{X}$, and therefore $ X _{\gamma_i} = \underline{X} _{\underline{\gamma}_i} $ and $ \tilde{X} _{\tilde{\gamma}_i} = \underline{\tilde{X}} _{\underline{\tilde{\gamma}}_i} $ for all $i$. It means that we can replace $X$ and $ \tilde{X} $ by $ \underline{X} $ and $ \underline{\tilde{X}} $ with no modification of the proof.

The only difference remaining is that in \cite{RaSe09}, they need the regeneration times $ \tau_i $ and $ \tilde{\tau}_i $ to have finite $p$-moments (which means they need $ p \leq p_0 $) to get the bound 
\begin{equation} \label{e:boundM}
 \PP _{ z , \tilde{z} } \left( \frac{m}{2n} < M_{\beta \wedge \tilde{\beta}} + h < \infty \right) \leq C \left( \frac{n}{m} \right) ^p  
 \end{equation}
for all $ z \cdot v_\star = \tilde{z} \cdot v_\star = 0 $, where $ M_{ \beta \wedge \tilde{\beta} } = \sup \{ X_i \cdot v_\star : i \leq \beta \wedge \tilde{\beta} \}$.

We will obtain this bound differently here, for all $p \geq 1$ and for $ M_{ \underline{\beta} \wedge \underline{\tilde{\beta}} } = \sup \{ \underline{X}_i \cdot v_\star : i \leq \underline{\beta} \wedge \underline{\tilde{\beta}} \}$. This is the reason why we had to introduce the walks $ \underline{X} $ and $ \underline{\tilde{X}} $: on the event $ \underline{\beta} \wedge \underline{\tilde{\beta}} < \infty $, thanks to our construction, we get
\begin{align*}
 M_{\underline{\beta} \wedge \underline{\tilde{\beta}}} + h 
&= \sup \{ \underline{X}_n \cdot v_\star : n \leq \underline{\beta} \wedge \underline{\tilde{\beta}} \} + h &\\
& \leq C \left( \sup_{ 0 \leq n \leq \underline{\tau}_1 }  \underline{X} _n \cdot v_\star + \sup_{ 0 \leq n \leq \underline{\tilde{\tau}}_1 }  \underline{\tilde{X}} _n \cdot v_\star \right) &\\
& \leq C \left( \sup_{ 0 \leq n \leq \underline{\tau}_1 } \| \underline{X} _n \| + \sup_{ 0 \leq n \leq \underline{\tilde{\tau}}_1 } \| \underline{\tilde{X}} _n \| \right) &\\
&= C \left( \sup_{ 0 \leq n \leq \tau_1 } \| X _n \| + \sup_{ 0 \leq n \leq \tilde{\tau}_1 } \| \tilde{X} _n \| \right) .&
\end{align*}
The first inequality holds because $ | \sup \{ \underline{X}_n \cdot v_\star : n \leq \underline{\beta} \wedge \underline{\tilde{\beta}} \} - \sup \{ \underline{\tilde{X}}_n \cdot v_\star : n \leq \underline{\beta} \wedge \underline{\tilde{\beta}} \} | \leq r_0 $ by construction. Then if $ \underline{\beta} \wedge \underline{\tilde{\beta}} = \underline{\beta} $, it means $ \underline{\beta} \wedge \underline{\tilde{\beta}} \leq \underline{\tau}_1 $ and we bound by the first term of the sum, else it means $ \underline{\beta} \wedge \underline{\tilde{\beta}} \leq \underline{\tilde{\tau}}_1 $ and we bound by the second term.

Markov's inequality and condition $(T)_\gamma$ then give us the equivalent of the bound (\ref{e:boundM}) for all $ p \geq 1 $. This concludes the proof.

\end{proof}

\begin{lem}
\label{lem:joint(T)}
For all $ p \geq 1 $,
\[ \sup _{ y \in \V_d } \EE _{0,y} \left[ \left( \sup_{0 \leq n \leq \mu_1} \|X_n\| \right) ^p \right] < \infty .\]
\end{lem}

\begin{proof}

Let $x,y \in \V_d$. Set $K$ the integer that satisfies $ \mu_1 = \tau _{K} $. By the triangle inequality, we get:
\[  \left( \sup_{0 \leq n \leq \mu_1} \|X_n\| \right) ^p 
\leq
\left( \sum _{k=1} ^K \sup_{ \tau _{k-1} \leq n \leq \tau _k } \|X_n\| \right) ^p  .\]
Holder's inequality gives:
\[  \left( \sup_{0 \leq n \leq \mu_1} \|X_n\| \right) ^p 
\leq
 K ^{p-1} \sum _{k=1} ^K  \left( \sup_{ \tau _{k-1} \leq n \leq \tau _k } \|X_n\| \right) ^p  .\]

Then
\begin{align*}
\EE _{0,y} \left[ \left( \sup_{0 \leq n \leq \mu_1} \|X_n\| \right) ^p \right]
& \leq  \EE _{0,y} \left[ \sum _{j=1} ^{+\infty} \mathbb{1}_{K=j} j ^{p-1} \sum _{k=1} ^j  \left( \sup_{ \tau _{k-1} \leq n \leq \tau _k } \|X_n\| \right) ^p  \right] & \\
& = \sum _{j=1} ^{+\infty} j ^{p-1} \EE _{0,y} \left[ \mathbb{1}_{K=j} \sum _{k=1} ^j  \left( \sup_{ \tau _{k-1} \leq n \leq \tau _k } \|X_n\| \right) ^p  \right] &\\
& \leq \sum _{j=1} ^{+\infty} j ^{p-1} \EE _{0,y} \left[ \mathbb{1}_{K=j} ^2  \right]^{\frac{1}{2}} \EE _{0} \left[ \left( \sum _{k=1} ^j  \left( \sup_{ \tau _{k-1} \leq n \leq \tau _k } \|X_n\| \right) ^p \right) ^2  \right]^{\frac{1}{2}} &
\end{align*}
where we used Cauchy-Schwarz in the last inequality.

Since, by H\"older's inequality, $ \left( \sum _{k=1} ^j  \left( \sup_{ \tau _{k-1} \leq n \leq \tau _k } \|X_n\| \right) ^p \right) ^2  \leq  j \sum _{k=1} ^j  \left( \sup_{ \tau _{k-1} \leq n \leq \tau _k } \|X_n\| \right) ^{2p} $, we get:
\begin{align*}
\EE _{0} \left[ \left( \sum _{k=1} ^j  \left( \sup_{ \tau _{k-1} \leq n \leq \tau _k } \|X_n\| \right) ^p \right) ^2  \right]
& \leq j \EE _{0} \left[ \sum _{k=1} ^j  \left( \sup_{ \tau _{k-1} \leq n \leq \tau _k } \|X_n\| \right) ^{2p} \right]  & \\
& \leq  C j^2 \EE _{0} \left[ \left( \sup_{ 0 \leq n \leq \tau _1 } \|X_n\| \right) ^{2p} \right] &\\
& \leq  C' j^2 &
\end{align*}
where $ C' < \infty $ thanks to condition~$(T)_\gamma$ and does not depend on $y$.

We then obtain
\begin{align*}
\EE _{0,y} \left[ \left( \sup_{0 \leq n \leq \mu_1} \|X_n\| \right) ^p \right]
& \leq \sum _{j=1} ^{+\infty} C j ^{p} \EE _{0,y} \left[ \mathbb{1}_{K=j}  \right]^{\frac{1}{2}} & \\
& = \sum _{j=1} ^{+\infty} C j ^{p} \PP _{0,y} \left[ K=j \right]^{\frac{1}{2}} .&
\end{align*}

As $ \mu_1 = \tau _{K} $, we get $ K \leq h \Lambda $. Lemma~\ref{lem:boundLambda} then allows us to conclude: for all $ q \geq 1 $
\[ \PP _{0,y} \left[ K=j \right] 
\leq \PP _{0,y} \left[ h \Lambda \geq j \right] 
\leq C_{q} (\frac{j}{h}) ^{-q},\]
where we recall that the bound is uniform on $ y \in \mathbb{V}_d $, and
\begin{align*}
\EE _{0,y} \left[ \left( \sup_{0 \leq n \leq \mu_1} \|X_n\| \right) ^p \right]
& \leq \sum _{j=1} ^{+\infty} C j ^{p} (\frac{j}{h}) ^{-\frac{q}{2}} & \\
& = \sum _{j=1} ^{+\infty} C' j ^{ p - \frac{q}{2} }  .&
\end{align*}

This sum is finite for $q$ big enough. This concludes the proof.

\end{proof}

\subsection{Markovian structure and coupling}

We will now show that for $ Y_i := \tilde{X} _{\tilde{\mu}_i} - X _{\mu_i} $, $ ( Y_i ) _{i\geq1} $ is a Markov process. Then we will construct a coupling to control its transitions.

\begin{pr}
\label{pr:Markov}
Set $ x,y \in \V_d $. Under $ \PP _{x,y} $, the process $ ( Y_i ) _{i\geq1} = ( \tilde{X} _{\tilde{\mu}_i} - X _{\mu_i} ) _{i\geq1} $ is a Markov chain on $\V_d$, with transition probabilities given by:
\[ q(x,y) = \PP _{0,x} \left( \tilde{X} _{\tilde{\mu}_1} - X _{\mu_1} = y | \beta = \tilde{\beta} = \infty \right) .\]
\end{pr}

The Markov chain only starts from $\tilde{X} _{\tilde{\mu}_1} - X _{\mu_1}$, because we do not know if $\beta = \tilde{\beta} = \infty$ after $X_0$ and $\tilde{X}_0$.

\begin{proof}
This proposition is very close from proposition~7.7 in \cite{RaSe09}. The proof follows exactly the same steps, as our modification of the regeneration structure preserves the independence property of the regeneration slabs.
\end{proof}

We now compare $Y_i$ to a random walk obtained similarly, but with the joint regeneration times of two independent walks in independent environments (instead of in the same environment). We consider a pair of walks $ ( X , \overline{X} ) $ of law $ \PP _0 \otimes \PP_z $, with $ z \in \V_d $. We denote by $\beta$ and $\overline{\beta}$ the backtracking times of $ X $ and $ \overline{X}$, and we construct the joint regeneration times $ ( \rho_i , \overline{\rho}_i ) _{i\geq1} $ in the same manner as we did for $ ( \mu_i , \tilde{\mu}_i ) _{i\geq1} $ (the only difference being that $ ( X , \tilde{X} ) $ were evolving in the same environment, whereas $ ( X , \overline{X} ) $ are in independent environments). Now, set $ \overline{Y}_i := \overline{X} _{\overline{\rho}_i} - X _{\rho_i} $.

\begin{pr}
The process $ ( \overline{Y}_i ) _{i\geq1} = ( \overline{X} _{\overline{\rho}_i} - X _{\rho_i} ) _{i\geq1} $ is a Markov chain on $\V_d$, and its transition probabilities satisfy:
\begin{align*}
\overline{q}(x,y) 
& = \PP _0 \otimes \PP _x  \left( \overline{X} _{\overline{\rho}_1} - X _{\rho_1} = y | \beta = \overline{\beta} = \infty \right) &\\
& = \PP _0 \otimes \PP _0  \left( \overline{X} _{\overline{\rho}_1} - X _{\rho_1} = y - x | \beta = \overline{\beta} = \infty \right) &\\
& = \overline{q}(0,y-x) &\\
& = \overline{q}(0,x-y) .&
\end{align*}
\end{pr}

\begin{proof}
This proposition is similar to proposition~7.8 of \cite{RaSe09}. The proof follows exactly the same steps, as our modification of the regeneration structure preserves the independence property of the regeneration slabs.
\end{proof}

As in lemma 7.9 of \cite{RaSe09}, this allows to prove that for all $ z, w $ such that $ q(z,w) > 0 $, we also get $ \overline{q}(z,w) > 0 $.

In the following, we will detach the notations $ Y_i $ and $ \overline{Y}_i $ from their definitions in terms of $ X, \tilde{X}, \overline{X} $. We will then use $ ( Y_i ) $ and $ ( \overline{Y}_i ) $ to represent the canonical Markov chains of transition probabilities $ q $ and $ \overline{q} $. This allow to construct the following coupling:

\begin{pr}
\label{pr:coupling}
The probability transitions $ q (x,y) $ for $Y$ and $ \overline{q} (x,y) $ for $\overline{Y}$ can be coupled in such a way that, for all $ x \in \V_d $, $x\neq0$, for all $p\geq1$
\[ \PP _{x,x} ( Y_1 \neq \overline{Y}_1 ) \leq C_p |x|^{-p} ,\]
where $C_p$ is a finite positive constant independent of $x$.
\end{pr}

\begin{proof}

We proceed as in the proof of proposition~7.10 in \cite{RaSe09} and construct a coupling of three walks $ ( X, \tilde{X}, \overline{X} ) $ such that the pair $ ( X, \tilde{X} ) $ has distribution $ \PP _{x,y} $ and the pair $ ( X, \overline{X} ) $ has distribution $ \PP _x \PP _y $.

For this, we take as before two independent walks $ ( X , \tilde{X} ) $ that evolve in a common environment $\omega$. We take another environment $\overline{\omega}$ independent of $\omega$, and construct the walk $\overline{X}$ as follows. We set $ \overline{X} _0 = \tilde{X} _0 $, then $\overline{X}$ moves according to the environment $ \overline{\omega} $ on the sites $ \{ X_k : 0 \leq k < \infty \} $ and according to the environment $ \omega $ on all other sites. Furthermore, $ \overline{X} $ is coupled to agree with $ \tilde{X} $ until the time $T = \inf \{ n \geq 0 : \overline{X}_n \in \{ X_k : 0 \leq k < \infty \} \} $ when it hits the path of $X$.

The details of the construction of this coupling, and the verification that $X$ and $\overline{X}$ are independent works exactly as in the proof of proposition~7.10 in \cite{RaSe09}, we will then omit this part here.

We then construct the joint regeneration times as before: $ ( \mu_1, \tilde{\mu}_1 ) $ for $ ( X , \tilde{X} ) $ and $ ( \rho_1, \overline{\rho}_1 ) $ for $ ( X, \overline{X} ) $. It allows to define the paths of the walks stopped at their respective joint regeneration times:
\[ ( \Gamma , \overline{\Gamma} ) := \left( ( X _{ 0, \mu_1 } , \tilde{X} _{ 0, \tilde{\mu}_1 } ) , ( X _{ 0, \rho_1 } , \overline{X} _{ 0, \overline{\rho}_1 } )  \right) .\]

Notice that when the sets $ X _{[ 0 , \mu_1 \vee \rho_1 )} $ and $ \tilde{X}_{[ 0 , \tilde{\mu}_1 )} \cup \overline{X}_{[ 0 , \overline{\rho}_1 )} $ are disjoint, we get by construction that the paths $ \overline{X} _{ 0 , \tilde{\mu}_1 \vee \overline{\rho}_1 } $ and $ \tilde{X} _{ 0 , \tilde{\mu}_1 \vee \overline{\rho}_1 } $ are identical. This implies $ ( \mu_1, \tilde{\mu}_1 ) = ( \rho_1, \overline{\rho}_1 ) $ and $ ( X _{\mu_1} , \tilde{X} _{\tilde{\mu}_1} ) = ( X _{\rho_1} , \overline{X} _{\overline{\rho}_1} ) $.

The following lemma gives us an estimate on this event. 
This is a point where our proof differs from the one in \cite{RaSe09}.

\begin{lem}
\label{lem:intersection}
For all $x,y$ such that $x-y \in \V_d$ and $ x \neq y $, for all $p\geq1$,
\[ \PP _{x,y} \left( X _{[ 0 , \mu_1 \vee \rho_1 )} \cap ( \tilde{X}_{[ 0 , \tilde{\mu}_1 )} \cup \overline{X}_{[ 0 , \overline{\rho}_1 )} )  \neq \emptyset \right) \leq C_p | x - y | ^{-p} .\]
\end{lem}

\begin{proof}
We get
\begin{align*}
& \PP _{x,y} \left( X _{[ 0 , \mu_1 \vee \rho_1 )} \cap ( \tilde{X}_{[ 0 , \tilde{\mu}_1 )} \cup \overline{X}_{[ 0 , \overline{\rho}_1 )} )  \neq \emptyset \right) &\\
&\leq  \PP _{x,y} \left( \sup_{0 \leq n \leq \mu_1} \|X_n - x\|  \vee  \sup_{0 \leq n \leq \rho_1} \|X_n - x\|  \vee  \sup_{0 \leq n \leq \tilde{\mu}_1} \|\tilde{X}_n - y\|  \vee  \sup_{0 \leq n \leq \overline{\rho}_1} \|\overline{X}_n - y\|  > \frac{|x-y|}{2} \right) .&
\end{align*}
Indeed, if the walk $X$ intersects with $\tilde{X}$, it also intersects with $\overline{X}$ because of the coupling. Such an intersection means that either the walk $X$ covered more than half the initial distance before $ \mu_1 \vee \rho_1 $, or $\tilde{X}$ and $\overline{X}$ did before $ \tilde{\mu}_1 $ respectively $ \overline{\rho}_1 $. 

Lemma~\ref{lem:joint(T)}, extended to cover the case of $ ( \rho_1 , \overline{\rho}_1 ) $, then give us the $ C_p | x - y | ^{-p} $ bound. It concludes the proof of the lemma.
\end{proof}

This proves that for all $p \geq 1$,
\[ \PP _{x,y} \left( \Gamma \neq \overline{\Gamma} \right) \leq C_p | x - y | ^{-p} .\]

The following of the proof (taking care of the conditioning on no backtracking) works again exactly as in the end of the proof of proposition~7.10 in \cite{RaSe09}, replacing their particular $p_0$ by any $p$. We will then omit this part here.

\end{proof}

\subsection{Bound on the number of common points}

We now return to the proof of the bound of $\EE\left[Q_n \right]$, where $Q_n$ represents the number of intersections of two independent copies of $X$ in the same random environment $\omega$ up to time $n$.

The use of the joint regeneration times allow us to write:
\[ \EE _{0,0} \left[ | X_{[0,n)} \cap \tilde{X}_{[0,n)} | \right] 
\leq
 \sum _{i=0} ^{n-1}  \EE _{0,0} \left[ | X_{[ \mu_i , \mu_{i+1} )} \cap \tilde{X}_{[ \tilde{\mu}_i , \tilde{\mu}_{i+1} )} | \right] .\]

The term $i=0$ is a finite constant thanks to lemma~\ref{lem:joint(T)}. Indeed, the number of common points is bounded by the number of points $y$ such that $ \| y \| \leq \sup_{0 \leq n \leq \mu_1} \|X_n\| $.

For each $ 0 < i < n $, we use the same decomposition into pairs of paths as in \cite{RaSe09}. It gives:
\[ \EE _{0,0} \left[ | X_{[ \mu_i , \mu_{i+1} )} \cap \tilde{X}_{[ \tilde{\mu}_i , \tilde{\mu}_{i+1} )} | \right]  
=
\sum _{x_1, y_1} \PP _{0,0} ( X _{\mu_i} = x_1 , \tilde{X} _{\tilde{\mu}_i} = y_1 ) \EE _{x_1, y_1} \left[ | X_{[0,\mu_1)} \cap \tilde{X}_{[0,\tilde{\mu}_1)} | | \beta = \tilde{\beta} = \infty \right] .\]
We can get bounds on this conditional expectation:
\begin{align*}
& \EE _{x_1, y_1} \left[ | X_{[0,\mu_1)} \cap \tilde{X}_{[0,\tilde{\mu}_1)} | | \beta = \tilde{\beta} = \infty \right] &\\
& \leq \eta ^{-1} \EE _{x_1, y_1} \left[  | X_{[0,\mu_1)} \cap \tilde{X}_{[0,\tilde{\mu}_1)} |  \right] &\\
& \leq C \EE _{x_1, y_1} \left[ \left( \sup_{0 \leq n \leq \mu_1} \| X_n - X_0 \| \right) ^d \mathbb{1}_{ X_{[0,\mu_1)} \cap \tilde{X}_{[0,\tilde{\mu}_1)} \neq \emptyset }  \right] &\\
& \leq C \EE _{x_1, y_1} \left[ \left( \sup_{0 \leq n \leq \mu_1} \| X_n - X_0 \| \right) ^{2d} \right] ^{\frac{1}{2}} \PP _{x_1, y_1} \left[  X_{[0,\mu_1)} \cap \tilde{X}_{[0,\tilde{\mu}_1)} \neq \emptyset  \right] ^{\frac{1}{2}} &\\
& \leq C' \left( ( 1 \vee | x_1 - y_1 | )^{-2p} \right) ^{\frac{1}{2}}  &\\
& = C' ( 1 \vee | x_1 - y_1 | )^{-p} &\\
& = h_p ( x_1 - y_1 ) &
\end{align*}
where we used successively lemma~\ref{lem:backtrack}, the bound on the number of common points by the number of points $y$ such that $ \| y \| \leq \sup_{0 \leq n \leq \mu_1} \|X_n\| $, the Cauchy-Scharz inequality and lemmas~\ref{lem:joint(T)} and \ref{lem:intersection}. On the last line, we used the definition:
\[ h_p (x) := C' ( 1 \vee | x | )^{-p} .\]

Inserting this in the precedent equalities gives: for any $ p \geq 1 $,
\begin{align*}
\EE _{0,0} \left[ | X_{[ \mu_i , \mu_{i+1} )} \cap \tilde{X}_{[ \tilde{\mu}_i , \tilde{\mu}_{i+1} )} | \right] 
& \leq \EE _{0,0} \left[ h ( \tilde{X} _{\tilde{\mu} _i} - X _{\mu _i} ) \right]  &\\
& = \sum _x \PP _{0,0} ( \tilde{X} _{\tilde{\mu} _1} - X _{\mu _1}  = x ) \sum _y q ^{i-1} (x,y) h_p(y) &
\end{align*}
where the last equality is obtained thanks to the Markov property of proposition~\ref{pr:Markov}.

We now want to use the following proposition:
\begin{pr}[Theorem~A.1 in \cite{RaSe09}]
Let $\mathbb{S}$ be a subgroup of $\Z ^d$. Set $ Y = (Y_k) _{k \geq 0} $ be a Markov chain on $\mathbb{S}$ with transition probabilities $ q(x,y) $. Set $ \overline{Y} = ( \overline{Y}_k ) _{ k \geq 0 } $ be a symmetric random walk on $\mathbb{S}$ with transition probabilities $ \overline{q} (x,y) = \overline{q} (y,x) = \overline{q} (0,y-x) $.

We make the following assumptions:
\begin{enumerate}
\item[(A.i)] The walk $\overline{Y}$ has a finite  third moment: $ \EE_0 ( | \overline{Y}_1 | ^3 ) < \infty $.
\item[(A.ii)] Set $ U_r := \inf \{ n \geq 0 : Y_n \notin [-r,r] ^d \} $ the time needed for the Markov chain $Y$ to exit a cube of size $2r+1$. Then there is a constant $ 0 < K < \infty $ such that for all $ r \geq 1 $,
\[ \sup _{ x \in [-r,r]^d } \EE _x (U_r) \leq K ^r .\]
\item[(A.iii)] Set $ Y = (Y ^1, \dots, Y^d) $. For every $ i \in \{ 1, \dots, d \} $, if the one-dimensional random walk $\overline{Y}^i$ is degenerate in the sense that $ \overline{q}(0,y) = 0 $ for $ y ^i \neq 0 $, then so is $Y^i$ in the sense that $ q ( x , y) = 0 $ whenever $ x^i \neq y^i $. It means that any coordinate that can move in the $Y$ chain somewhere in space can also move in the $\overline{Y}$ walk.
\item[(A.iv)] For all $ x \neq 0 $, we can couple the transition probabilities $q$ and $\overline{q}$ to satisfy: for all $ p \geq 1 $,
\[ \PP _{x,x} ( Y_1 \neq \overline{Y}_1 ) \leq C |x| ^{-p} ,\]
with $ 0 < C < \infty $ independent of $x$.
\end{enumerate}

Now take $h$ a function on $\mathbb{S}$ such that for $ p_0 \geq 1$, $ 0 \leq h(x) \leq C ( 1 \vee |x| ) ^{-p_0} $ with $C$ a finite positive constant.

Then there are constants $ 0 < C < \infty $ and $ 0 < \eta < \frac{1}{2} $ such that for all $ n \geq 1 $ and $ z \in \mathbb{S} $,
\[ \sum _{k=0} ^{n-1} \EE _z ( h (Y_k) ) 
= \sum _y h(y) \sum _{k=0} ^{n-1} \PP _z ( Y_k = y ) 
\leq C n ^{1 - \eta} .\]
Furthermore, $ 1 - \eta $ can be taken arbitrarily close to $\frac{1}{2}$ if we take $p_0$ big enough.

\end{pr}

We will prove that our Markov Chains satisfy assumptions $(A.i), (A.ii), (A.iii)$ and $(A.iv)$ of this theorem. Assumption $(A.i)$ follow from lemma~\ref{lem:joint(T)}, which gives for all $p$: $ \EE_{0,x} ( | \overline{X} _{\overline{\rho}_k } |^p ) + \EE_{0,x} ( | X _{\rho_k } |^p ) < \infty $. It implies $ \EE_0 ( | \overline{Y}_1 | ^3 ) < \infty $ as needed. Assumption $(A.iii)$ follow from the fact that for all $ z, w $ such that $ q(z,w) > 0 $, we also get $ \overline{q}(z,w) > 0 $. Assumption $(A.iv)$ is directly deduced from proposition~\ref{pr:coupling}.

It only remains to check assumption $(A.ii)$. We proceed as in lemma~7.13 of \cite{RaSe09}. Their proof (including their Appendix C) remains unchanged by our new definition of regeneration times.

As $ h_p $ also satisfies the hypothesis of this theorem (for $p_0 = p$), we get constants $  0 < C < \infty $ and $ 0 < \eta < \frac{1}{2} $ such that for all $ x \in \V_d $ and $ n \geq 1 $,
\[ \sum _{i=1} ^{n-1} \sum _y q ^{i-1} (x,y) h_p(y) \leq C n ^{1 - \eta} .\]
Inserting this back in the previous inequalities, we finally get: for all $p \geq 1 $, we get constants $  0 < C < \infty $ and $ 0 < \eta < \frac{1}{2} $ such that
\[  \EE\left[Q_n \right] \leq C n ^{1 - \eta} \;\; \forall \; n \in \N.\] 
Since this holds for any $p \geq 1$, the constant $1 - \eta$ can be made as close to $\frac{1}{2}$ as desired.
This concludes the proof of theorem~\ref{theo:boundonintersections}.


\begin{thebibliography}{99}

\bibitem{BeDrRa12}
N. Berger, A. Drewitz and A.F. Ram\'\i rez,
 Effective polynomial ballisticity condition for random walk in random environment, to appear in Comm. Pure Appl. Math.

\bibitem{BeZe08}
N. Berger and O. Zeitouni,
A quenched invariance principle for certain ballistic random walks in i.i.d. environments,
In and out of equilibrium. 2, Birkh\"auser, 2008, 60, 137-160.

\bibitem{BoSz02}
E. Bolthausen and A.-S. Sznitman,
On the static and dynamic points of view for certain random walks in random environment, Methods Appl. Anal., 2002, 9, 345-375.

\bibitem{BoRaSa14}
\'E. Bouchet, A. F. Ram\'irez and C. Sabot,
Sharp ellipticity conditions for ballistic behavior of random walks in random environment,
Preprint, arXiv:1310.6281, 2013.

\bibitem{CaRa13}
D. Campos and A. F. Ram\'irez,
Ellipticity criteria for ballistic behavior of random walks in random environment,
Probability Theory and Related Fields, Springer Berlin Heidelberg, 2013, 1-63.


\bibitem{EnSa06}
N. Enriquez and C. Sabot,
Random walks in a Dirichlet environment,
Electron. J. Probab., 2006, 11, no. 31, 802-817.

\bibitem{FrKi14}
A. Fribergh and D. Kious, Local trapping for elliptic random walks in random environments in $\mathbb{Z}^d$, http://arxiv.org/abs/1404.2060


\bibitem{Pe88}
R. Pemantle,
Phase transition in reinforced random walk and RWRE on trees,
Ann. Probab., 1988, 16, 1229-1241.

\bibitem{RaSe09}
F.\ Rassoul-Agha and T.\ Sepp\"al\"ainen,
Almost sure functional central limit theorem for ballistic random walk in random environment,
Ann.\ Inst.\ H.\ Poincar\'e Probab.\ Statist.\ 45 (2009) 373--420.

\bibitem{Sa11}
C. Sabot,
Random walks in random Dirichlet environment are transient in dimension $d\geq 3$. 
Probab. Theory Related Fields 151 (2011), no. 1-2, 297-317.

\bibitem{Sa13}
C. Sabot,
Random Dirichlet environment viewed from the particle in dimension $d \geq 3$
Ann. Probab., 2013, 41, 722-743.

\bibitem{Si07}
F. Simenhaus,
Asymptotic direction for random walks in random environments,
Ann. Inst. Henri Poincar\'e Probab. Stat., 2007, 43, 751-761.

\bibitem{SzZe99}
A.-S. Sznitman and M. Zerner,
A law of large numbers for random walks in random environment,
Ann. Probab., 1999, 27, 1851-1869.

\bibitem{Sz00}
A.-S. Sznitman,
Slowdown estimates and central limit theorem for random walks in random environment,
J. Eur. Math. Soc. (JEMS), 2000, 2, 93-143.

\bibitem{Sz01}
A.-S. Sznitman,
On a class of transient random walks in random environment,
Ann. Probab., 2001, 29, 724-765.

\bibitem{Sz02}
A.-S. Sznitman,
An effective criterion for ballistic behavior of random walks in random environment,
Probab. Theory Relat. Fields 122, 509-544 (2002).

\bibitem{To09}
L. Tournier,
Integrability of exit times and ballisticity for random walks in Dirichlet environment, Electron. J. Probab., 2009, 14, no. 16, 431-451.

\bibitem{Ze02}
M. Zerner,
A non-ballistic law of large numbers for random walks in i.i.d. random environment,
Electron. Comm. Probab., 2002, 7, 191-197.

\end{thebibliography}
\end{document}